\documentclass[12pt]{amsart}

\usepackage{amssymb}
\def\colon{{:}\;}

\def\R {\Bbb R}

\newcommand{\beq}{\begin{equation}}
\newcommand{\eeq}{\end{equation}}

\newcommand{\ben}{\begin{eqnarray}}
\newcommand{\een}{\end{eqnarray}}

\newcommand{\bet}{\begin{eqnarray*}}
\newcommand{\eet}{\end{eqnarray*}}

\newtheorem{thm}{Theorem}[section]

\newtheorem{pro}[thm]{Proposition}
\newtheorem{cor}[thm]{Corollary}

\newtheorem{de}[thm]{Definition}
\newtheorem{re}[thm]{Remark}

\newtheorem{que}[thm]{Question}

\def\R {\mathbb R}
\def\N {\mathbb N}
\def\M {{\mathcal M}}
\def\F {{\mathcal F}}

\def\I{{\mathcal I}}

\def\C{{\mathbb C}}

\begin{document}
\baselineskip 16pt

\title[Equilibrium states of the pressure function]
{Equilibrium states of the pressure function for products of matrices}

\author{De-Jun Feng}
\address{
Department of Mathematics\\
The Chinese University of Hong Kong\\
Shatin,  Hong Kong\\
P.\ R.\ China
}
\email{djfeng@math.cuhk.edu.hk}

\author{Antti K\"{a}enm\"{a}ki}
\address{
Department of Mathematics and Statistics\\
P.O.\ Box 35 (MaD)\\
FI-40014 University of Jyv\"askyl\"a\\
Finland
}
\email{antti.kaenmaki@jyu.fi}
\date{\today}

%{\small

\begin{abstract}
Let  $\{M_i\}_{i=1}^\ell$ be a non-trivial family of $d\times d$  complex matrices, in the sense that for any $n\in \N$, there exists $i_1\cdots i_n\in \{1,\ldots, \ell\}^n$ such that $M_{i_1}\cdots M_{i_n}\neq {\bf 0}$.  Let  $P \colon (0,\infty)\to \R$ be the pressure function of  $\{M_i\}_{i=1}^\ell$.   We  show that for each $q>0$, there are at most $d$ ergodic $q$-equilibrium states of $P$, and each of them  satisfies certain Gibbs property.

\end{abstract}

\thanks { }

\maketitle
\setcounter{section}{0}

\section{Introduction and results}
\setcounter{equation}{0}

In this paper, we study the thermodynamic formalism for matrix products. We will  characterize the structure of equilibrium states of pressure functions, and  also examine the Gibbs properties of such states. This work was first carried out in \cite{FeLa02} in the case that the involved matrices are non-negative and satisfy a kind of irreducibility. Some applications were given in the multifractal analysis of the top Lyapunov exponents of matrix products \cite{FeLa02, Fen03a, Fen09} (see also \cite{FeHu09}). In this paper, we will consider arbitrary complex matrices.

Let $(\Sigma,\sigma)$ be the one-sided full shift over the alphabet $\{1,\ldots,\ell\}$ (cf. \cite{Bow75}) and let $\{M_i\}_{i=1}^\ell$ be a family of $d\times d$ complex matrices. For $q>0$, we define
\begin{equation}
\label{e-1.1}
P(q)=\lim_{n\to \infty}\frac{1}{n}\log \sum_{J\in \Sigma_n}\|M_J\|^q,
\end{equation}
where $\Sigma_n$ is the collection of all words of length $n$ over $\{1,\ldots,\ell\}$,
$M_J=M_{j_1}\cdots M_{j_n}$ for $J=j_1\cdots j_n$, and $\|\cdot\|$ is the standard matrix norm. By sub-additivity, the above limit exists and  takes  values in the set $\R\cup\{-\infty\}$. The function $P$ is called the {\it pressure function} of $\{M_i\}_{i=1}^\ell$. It plays an important role in the multifractal analysis of Lyapunov exponents of matrices \cite{FeLa02, Fen03a, Fen09}. Moreover, it is closely related to the dimension theory of  self-affine sets and measures \cite{Fal88, Kae04}.

Denote the collection of all $\sigma$-invariant Borel probability measures on $\Sigma$ by $\M_\sigma(\Sigma)$. Endow $\M_\sigma(\Sigma)$ with the weak-star topology. For $\mu\in \M_\sigma(\Sigma)$, we define
\begin{equation}
\label{e-1.2}
M_*(\mu)=\lim_{n\to \infty}\frac{1}{n}\sum_{J\in \Sigma_n}\mu([J])\log \|M_J\|,
\end{equation}
where $[J]$ denotes the $n$-th cylinder $\{x=(x_i)_{i=1}^\infty\in \Sigma:\; x_1\cdots x_n=J\}$ in $\Sigma$.
The term $M_*(\mu)$ is called the {\it Lyapunov exponent of $\{M_i\}_{i=1}^\ell$ with respect to $\mu$}. It also takes  values in the set $\R\cup\{-\infty\}$.
The following variational principle for $P$ was proved in \cite{CFH08} in a more general sub-additive setting:
\begin{equation}
\label{e-1.3}
P(q)=\sup\{qM_*(\mu)+h(\mu):\; \mu\in \M_\sigma(\Sigma)\},
\end{equation}
where $h(\mu)$ denotes the measure-theoretic entropy of $\mu$ with
respect to $\sigma$ (cf. \cite{Wal-book}). We remark that  \eqref{e-1.3} was  proved earlier  in \cite{Fen04, Kae04} when the matrices are non-negative or  invertible, respectively. For given $q>0$, let
\begin{equation}
\label{e-1.4}
\I_q=\{\mu\in \M_\sigma(\Sigma):\;P(q)=qM_*(\mu)+h(\mu)\}.
\end{equation}
Each element $\mu$ in $\I_q$ is called a {\it $q$-equilibrium state of $P$}. Since both $M_*(\cdot)$ and $h(\cdot)$ are upper semi-continuous on $\M_\sigma(\Sigma)$,   $\I_q$ is a non-empty closed convex subset of $\M_\sigma(\Sigma)$.
In particular, $\I_q$ contains ergodic elements (each extreme point of $\I_q$ is an ergodic measure).

Our main purpose   is to characterize the structure of $\I_q$. This question was partially  raised from \cite{KV09}. A complete characterization is given in Theorem \ref{thm-1.2}. In the following, we shall present the setting and results. Proofs of the results are postponed until \S \ref{sec-2}.

\begin{de}
{\rm
Let ${\Bbb F}$ be $\R$ or $\C$.  A family of $d\times d$ matrices $\{M_i\}_{i=1}^\ell$ with entries in ${\Bbb F}$ is said to be {\it irreducible over ${\Bbb F}^d$} if there is no non-zero proper linear subspace $V$ of ${\Bbb F}^d$ such that
$M_iV\subseteq V$ for all $i \in \{1,\ldots,\ell\}$.
}
\end{de}

The above definition is adopted  from \cite[p.\ 48]{BoLa85}. If $\{M_i\}_{i=1}^\ell$ is irreducible over ${\Bbb F}^d$, then there exist
$D>0$ and $k\in \N$ such that for any words $I,J\in
\Sigma^*=\bigcup_{n=1}^\infty \{1,\ldots,\ell\}^n$, there exists a
word $K$ in $\bigcup_{n=1}^k \{1,\ldots,\ell\}^n$ such that
\begin{equation}
\label{e-2.1} \|M_{IKJ}\|\geq D\|M_I\| \|M_J\|.
\end{equation}
For a proof, see \cite[Proposition 2.8]{Fen09}. This property is crucial
in the proof of the following proposition.

\begin{pro}
\label{pro-1.1} Let ${\Bbb F}$ be $\R$ or $\C$, and $\{M_i\}_{i=1}^\ell$ a family of $d\times d$ matrices with entries in ${\Bbb F}$. If $\{M_i\}_{i=1}^\ell$ is irreducible over ${\Bbb F}^d$, then for each
$q>0$, $P$ has a unique $q$-equilibrium state $\mu_q$. Furthermore,
$\mu_q$ has the following Gibbs property:
\begin{equation}
\label{e-1.5}
C^{-1}\exp(-nP(q))\|M_J\|^q \leq \mu_q([J]) \leq C\exp(-nP(q))\|M_J\|^q
\end{equation}
for all $n\in \N$ and $J\in \Sigma_n$. Moreover, $P$ is
differentiable over $(0,\infty)$ and $P'(q)=M_*(\mu_q)$ for $q>0$.
\end{pro}

\begin{re}
{\rm     Proposition \ref{pro-1.1} is an analogue of Bowen's theory about the equilibrium state of H\"{o}lder continuous additive potentials (cf.\ \cite{Bow75}). See
\cite{Rue78, Wal-book} for backgrounds and  more details about the classical thermodynamic formalism of additive potentials.  Proposition \ref{pro-1.1}  was first proved in \cite{FeLa02} for
non-negative matrices under a different irreducibility assumption (that is, there exists $r\in \N$ so that $\sum_{i=1}^r(M_1+\cdots+ M_\ell)^r$ is a strictly positive matrix). An extension was recently given in \cite[Theorem 5.5]{Fen09b} to certain sub-additive potentials.
}
\end{re}

Let us next consider the non-irreducibility case. Denote
the $n\times m$ zero matrix by ${\bf 0}_{n\times m}$.

\begin{pro}
\label{pro-0} Let ${\Bbb F}$ be $\R$ or $\C$, and $\{M_i\}_{i=1}^\ell$ a family of $d\times d$ matrices with entries in ${\Bbb F}$. Then there exist an invertible $d\times d$ matrix $T$,
%a positive integer $1 \le t \le d$,
$t \in \{1,\ldots,d\}$,
and positive integers $d_1,\ldots,d_t$ with $d=d_1+\cdots +d_t$ such that
for every $i \in \{ 1,\ldots,\ell \}$ the product
$T^{-1}M_iT$ is a partitioned matrix of the
form
\begin{equation}
\label{e-1.6} T^{-1}M_iT=\left( A^{(j,k)}_i\right)_{1 \leq j,k \leq t},
\end{equation}
where $A^{(j,k)}_i$, $j,k \in \{ 1,\ldots,t \}$, satisfy the following two properties:
\begin{itemize}
\item[(i)]  $A^{(j,k)}_i$ is a $d_j \times d_k$ matrix and $A^{(j,k)}_i={\bf 0}_{d_j\times d_k}$ when $j>k$.
\item[(ii)] For any $j \in \{ 1,\ldots,t \}$, either the family $\{A^{(j,j)}_i\}_{i=1}^\ell$ is irreducible over ${\Bbb F}^{d_j}$, or
$A^{(j,j)}_i={\bf 0}_{d_j\times d_j}$ for all $i \in \{1,\ldots,\ell\}$.
\end{itemize}
\end{pro}

Considering the partition (\ref{e-1.6}) in the above proposition, we
set
$$\Lambda=\Lambda(\{M_i\}_{i=1}^\ell)=\{j\in \{1,\ldots, t\}:\; \{A_i^{(j,j)}\}_{i=1}^\ell
\mbox{ is irreducible over ${\Bbb F}^{d_j}$}\}.
$$

\begin{re} {\rm
It is possible that $\Lambda=\emptyset$. For instance, this is the
case for $\{M_i\}_{i=1}^2$, where
$$
M_1=\left(\begin{array}{ll} 0 & 1\\
0&0
\end{array}
\right), \qquad M_2=\left(\begin{array}{ll} 0 & 2\\
0&0\end{array} \right).
$$
Anyhow, it holds that $\Lambda=\emptyset$ if and only if there is
$k\in \N$ such that $M_{i_1}\cdots M_{i_n}={\bf 0}_{d\times d}$ for
all $n\geq k$ and $i_1\cdots i_n\in \{1,\ldots, \ell\}^n$.
Observe first that $T^{-1}M_{i_1} \cdots M_{i_n}T$ is a partitioned
matrix of the form $(B^{(j,k)})_{1\le j,k \le t}$, where
\begin{equation} \label{e-2.7_1}
\begin{split}
  B^{(j,k)}&=\sum_{1\leq y_1,\ldots,y_{n-1}\leq t} A_{i_1}^{(j,y_1)}A_{i_2}^{(y_1,y_2)}\cdots
  A_{i_n}^{(y_{n-1},k)}\\
  &=\sum_{j\leq y_1\leq y_2\leq \cdots\leq y_{n-1}\leq k} A_{i_1}^{(j,y_1)}A_{i_2}^{(y_1,y_2)}\cdots
  A_{i_n}^{(y_{n-1},k)}
\end{split}
\end{equation}
is a $d_j \times d_k$ matrix.
According to (ii) of Proposition \ref{pro-0}, $\Lambda = \emptyset$
implies $A^{(j,j)}_i = {\bf 0}_{d_j \times d_j}$
for all $i \in \{ 1,\ldots,\ell \}$ and $j \in \{ 1,\ldots,t \}$.
Hence $M_{i_1} \cdots M_{i_n} = {\bf 0}_{d \times d}$ for all $n>t$
by \eqref{e-2.7_1}.
To see the converse, assume contrarily that $\{ A^{(j,j)}_i \}_{i=1}^\ell$
is irreducible over ${\Bbb F}^{d_j}$ for some $j \in \{ 1,\ldots,t \}$.
It follows now from \eqref{e-2.1} that for every $n \in \N$ there exists
a word $i_1 \cdots i_n$ such that $A^{(j,j)}_{i_1} \cdots A^{(j,j)}_{i_n}
\ne {\bf 0}_{d_j \times d_j}$ and, consequently, $M_{i_1}\cdots M_{i_n}
\ne {\bf 0}_{d\times d}$.
}
\end{re}

\begin{de}

{\rm A family of $d\times d$ complex matrices $\{M_i\}_{i=1}^\ell$ is called {\it non-trivial} if $\Lambda\neq \emptyset$, or equivalently,  for each $n\in \N$, there exists
$I\in \{1,\ldots, \ell\}^n$ such that $M_I\neq {\bf 0}_{d\times d}$.}
\end{de}

In the following, we always assume that $\{M_i\}_{i=1}^\ell$ is non-trivial.   If $j\in \Lambda$, then the pressure
function of $\{A^{(j,j)}_i\}_{i=1}^\ell$ is denoted by $P_j$ and
the Lyapunov exponent of $\{A^{(j,j)}_i\}_{i=1}^\ell$ with respect
to $\mu$ is denoted by $A^{(j)}_{*}(\mu)$. The following is
the main result of our paper.

\begin{thm}
\label{thm-1.2}
In the above general setting, it holds that
\begin{itemize}
\item[(i)] $M_*(\mu)=\max\{A_{*}^{(j)}(\mu):\; j\in \Lambda\}$ for each ergodic measure $\mu\in \M_\sigma(\Sigma)$.
\item[(ii)]
 $P$
is a real-valued convex function on $(0,\infty)$, and
$P(q)=\max\{P_j(q):\; j\in \Lambda\}$ for all $q>0$.
\item[(iii)]
if $q>0$ and $\mu_{j,q}$, $j\in \Lambda$, is the unique
$q$-equilibrium state for $P_j$, then
\begin{equation*}
\I_q=\mbox{\rm conv}\{\mu_{j,q}:\; P_j(q)=P(q)\},
\end{equation*}
where $\mbox{\rm conv}(A)$ is the convex hull of $A$.
\end{itemize}
\end{thm}

\begin{re}
\label{re-a} {\rm
The equality in (i) of Theorem \ref{thm-1.2} may fail for
non-ergodic measures of $\M_\sigma(\Sigma)$.
For instance, consider $\{M_i\}_{i=1}^2$, where $M_1={\rm
diag}(1,2)$ and $M_2={\rm diag}(3,2)$. Let $\mu_1=\delta_{1^\infty}$,
$\mu_2=\delta_{2^\infty}$ (here $\delta_x$ denotes the Dirac measure at $x$), and $\mu=p\mu_1+(1-p)\mu_2$ for some
$0<p<1$. It is easy to check that
$$
M_*(\mu_1)=\log 2,\quad A_*^{(1)}(\mu_1)=0,\quad A_*^{(2)}(\mu_1)=\log 2
$$
and
$$
M_*(\mu_2)=\log 3,\quad A_*^{(1)}(\mu_2)=\log 3,\quad A_*^{(2)}(\mu_2)=\log 2.
$$
Since $M_*(\cdot)$, $A^{(1)}_*(\cdot)$, and $A^{(2)}_*(\cdot)$ are
affine on $\M_\sigma(\Sigma)$, we have
$$M_*(\mu)=p\log 2+(1-p)\log
3, \quad A^{(1)}_*(\mu)=(1-p)\log 3,\quad A^{(2)}_*(\mu)=\log 2,$$
and thus, $M_*(\mu)> \max\{A_{*}^{(i)}(\mu):\; i\in \{1,2\}\}$.
}

\end{re}

\begin{re} {\rm
The pressure function for products of matrices has been studied in the literature under some stronger conditions.
Let $\{M_i\}_{i=1}^\ell$ be a family of real invertible matrices. Assume that $\{M_i\}_{i=1}^\ell$ satisfies  the strong irreducibility and contraction conditions (cf.\ \cite{BoLa85, GuLe04}).  Guivarc'h and Le Page showed in
\cite[Theorem 8.8]{GuLe04} that the pressure function $P$ of $\{M_i\}_{i=1}^\ell$ corresponds to the logarithm  of the spectral radius of certain Ruelle transfer operator and moreover, $P$ is real analytic on $(0,\infty)$, and it can be extended to an analytic function on $\{z\in \C:\; \Re z>0\}$.
This strengthens an early result of Le Page \cite{LeP82}.
}

\end{re}

\section{Proofs of the results} \label{sec-2}
\setcounter{equation}{0}

This section is dedicated to the proof of Theorem \ref{thm-1.2}. For the convenience of the reader we shall also present complete proofs for Propositions \ref{pro-1.1} and \ref{pro-0}.

\begin{proof}[Proof of Proposition  \ref{pro-1.1}]
Let $q>0$. Define a sequence of probability measures
$(\nu_{n,q})_{n\geq 1}$ on $\Sigma$ so that
$$
\nu_{n,q}([I])=\frac{\|M_I\|^q}{\sum_{J\in \Sigma_n}
\|M_J\|^q}
$$
for all $I \in \Sigma_n$.
Let $\nu_q$ be a limit point of the sequence $(\nu_{n,q})_{n \geq 1}$
in the weak topology. Furthermore, let $\mu_q$ be a limit point of
the sequence
$$\left( \frac{1}{n}\sum_{j=0}^{n-1} \nu_q\circ \sigma^{-j} \right)_{n \geq 1}$$ in the weak topology. Using (\ref{e-2.1}) and a
proof essentially identical to that of \cite[Theorem 3.2]{FeLa02},
we see that  $\mu_q \in \M_\sigma(\Sigma)$ is ergodic and has the Gibbs
property (\ref{e-1.5}). Thus
% Entropy can be calculated by using cylinders:
\begin{equation*}
\begin{split}
  qM_*(\mu_q) + h(\mu_q) &\ge \lim_{n \to \infty}\frac{1}{n} \sum_{J \in \Sigma_n} \mu_q([J])\log\bigl( C^{-1}\exp(nP(q))\mu_q([J]) \bigr) \\ &\qquad- \lim_{n \to \infty}\frac{1}{n} \sum_{J \in \Sigma_n} \mu_q([J])\log\mu_q([J]) = P(q).
\end{split}
\end{equation*}
Recalling \eqref{e-1.3}, this implies $\mu_q \in \I_q$.
%Thus by Shannon-Breiman-McMillan Theorem, we have for $\mu_q$-a.e.
%$x=(x_i)_{i=1}^\infty \in \Sigma$,
%\begin{equation*}
%\begin{split}
%  h(\mu_q)=&-\lim_{n\to \infty}\frac{\log \mu_q([x_1\ldots x_n])}{n}\\
%  &=P(q)-\lim_{n\to \infty} \frac{q\log \|M_{x_1\ldots
%  x_n}\|}{n}\qquad \mbox{(by (\ref{e-1.5}))}\\
%  &=P(q)-qM_*(\mu_q),
%\end{split}
%\end{equation*}
%where the last step follows from Kingman's subadditive ergodic theorem.
%That is, $P(q)=h(\mu_q)+qM_*(\mu_q)$. Hence, $\mu_q \in \I_q$.

Applying (\ref{e-1.5}) and the ergodicity of $\mu_q$, and using an identical argument as in \cite[proof of Theorem 1.22]{Bow75} (or using \cite[Theorem 3.6]{KV09}), we see that $\mu_q$ is the
unique element in $\I_q$. According to this uniqueness, we have
$P'(q)=M_*(\mu_q)$, which  follows from  the Ruelle-type derivative
formula of pressures obtained in \cite[Theorem 1.2]{Fen04}:
$$
P'(q-)=\inf\{\M_*(\mu): \; \mu\in \I_q\},\quad
P'(q+)=\sup\{\M_*(\mu): \; \mu\in \I_q\}.
$$
We remark that  although \cite[Theorem 1.2]{Fen04} only deals with non-negative matrices, the proof given there works for arbitrary matrices. Alternatively,
to show that $P'(q) = M_*(\mu_q)$, we may apply \eqref{e-1.5} and the ergodicity of
$\mu_q$, and follow \cite[proof of Theorem 2.1]{Heu98} (see also \cite[Theorem 4.4]{KV09}).
\end{proof}

\begin{proof}[Proof of Proposition \ref{pro-0}]
We prove the proposition by induction on $d$. Clearly the
proposition  is true when $d=1$. Assuming there exists an integer
$p$ so that the proposition is true for all $d\leq p$, we show below
that it remains true for $d=p+1$. Let $L(n,m)$ be the collection of
all $n\times m$ matrices with entries in ${\Bbb F}$.

If $\{M_i\}_{i=1}^\ell$ is irreducible over ${\Bbb F}^d$, we
simply take $t=1$ and have nothing else to prove.  We may thus
assume that $\{M_i\}_{i=1}^\ell$ is reducible, that is,
there exists a non-zero proper linear space $V$ of ${\Bbb F}^d$ such that
$M_iV\subset V$. If we let $v=\dim V$, then $1\leq v$ and $d-v\leq d-1=p$.
We choose an invertible linear map $T_1:\; {\Bbb F}^d\to {\Bbb F}^d$ such that
$T_1({\Bbb F}^v\times \{ 0\})=V$. Then for each $i\in\{1,\ldots,\ell\}$
there exist $E_i\in L(v,v)$, $B_i\in L(v,  d-v)$, $D_i\in L(d-v, d-v)$
so that
$$
T_1^{-1}M_iT_1=\left(\begin{array}{ll}
 E_i  & B_i\\
{\bf 0}_{(d-v)\times v} & D_i
\end{array}
 \right).
$$
Now by the induction assumption, there exist invertible
matrices $T_2\in L(v,v)$ and $T_3\in L(d-v, d-v)$  such that
$(T_2^{-1}E_iT_2)_{i=1}^\ell$ and $(T_3^{-1}D_iT_3)_{i=1}^\ell$ have
the desired partitioned form for all $i \in \{ 1,\ldots,\ell \}$. It
follows that
$$
T_4=T_1\left(
\begin{array}{ll} T_2 & {\bf 0}_{v\times (d-v)}\\
{\bf 0}_{(d-v)\times v} & T_3
\end{array}\right)
$$
is an invertible $d \times d$ matrix and
$$T_4^{-1}M_iT_4= \left(
\begin{array}{ll}
T_2^{-1}E_iT_2 & T_2^{-1} B_i T_3\\
{\bf 0}_{(d-v)\times v} & T_3^{-1}D_iT_3
\end{array}
\right)
$$
has the desired partitioned form for all $i \in \{ 1,\ldots,\ell \}$.
\end{proof}

Before proving Theorem \ref{thm-1.2}, we shall first prove the following auxiliary result.

\begin{pro}
\label{lem-2.1} Let $(X,\F, \mu)$ be a probability space and  $T:\;X\to X$ an ergodic measure-preserving  transformation.
Let $\{f_n\}_{n=1}^\infty$ be a sequence of non-negative Borel measurable functions on $X$ such that $\sup_{x\in X}f_1(x)<\infty$ and
\begin{equation}
\label{e-pro}
f_{n+m}(x)\leq f_m(x)f_n(T^m x)
\end{equation}
for all $m,n\in \N$ and $x\in X$. If
$\epsilon>0$ and  $\alpha=\lim_{n\to \infty}({1}/{n})\int \log f_n \; d\mu$, then the following claims hold:
\begin{itemize}
\item[(i)] If $\alpha\neq -\infty$, then for $\mu$-almost every $x\in X$, there exists a positive integer $n_0(x)$ such that
\begin{equation}
\label{e-2.2} |\log f_n(T^m x)-n\alpha|\leq (n+m)\epsilon
\end{equation}
for all $n\geq n_0(x)$ and $m\in \N$.
\item[(ii)]
If $\alpha=-\infty$, then for any $N>0$ and  $\mu$-almost every $x\in X$, there exists a positive integer $n_0(x)$ such that
\begin{equation}
\label{e-2.5}
\log f_n(T^m x)\leq -Nn+ (n+m)\epsilon
\end{equation}
for all $n\geq n_0(x)$ and $m\in \N$.
\end{itemize}
\end{pro}
\begin{proof}
We only prove (i). The proof of (ii) is similar.

Assume that $\alpha\in \R$. Let $\epsilon>0$ and take $0<\delta<\epsilon/4$. By the Kingman's sub-additive ergodic theorem,
for $\mu$-almost every $x\in X$, there exists $n_0(x)$ such that
$$
|\log f_n(x)-n\alpha|\leq n\delta
$$
for all $n \geq n_0(x)$, and
$$
|\log f_m(x)-m\alpha|\leq (n_0(x)+m)\delta
$$
for all $m\in \N$.
Hence by \eqref{e-pro}, we have for $n\geq n_0(x)$ and $m \in \N$,
\begin{equation}
\label{e-t}
\begin{split}
\log f_n(T^m x)&\geq \log f_{n+m}(x) - \log f_m(x) \\
&\geq (n+m)(\alpha-\delta)-m(\alpha+\delta)-n_0(x)\delta
\\
&\geq n\alpha-2(n+m)\delta\geq n\alpha-(n+m)\epsilon.
\end{split}
\end{equation}

To see the opposite inequality, take $k$ large enough such that $|\beta-\alpha|<\delta$, where
$$
\beta=\frac{1}{k}\int \log f_k \; d\mu.
$$
Since $\{f_n(x)\}_{n=1}^\infty$ is sub-multiplicative, by \cite[Lemma 2.2]{CFH08}, we have for any $n\geq 2k$,
$$
(f_n(x))^k\leq C^{2k^2} \prod_{j=0}^{n-k}f_k(T^j x)
$$
for all $x\in X$,
where $C=\max\{1,\sup_{x\in \Sigma} f_1(x)\}$.
It follows that for $n\geq 2k$ and $m\in \N$,
$$
\log f_n(T^m x) \leq 2k \log C +\sum_{i=0}^{n-k+m}\frac{1}{k}\log f_k(T^i x)-\sum_{i=0}^{m-1}\frac{1}{k}\log f_k(T^i x).
$$
Applying the Birkhoff ergodic theorem to the function $\frac{1}{k}\log f_k$, and combining it with the above inequality, we see that  for $\mu$-almost every $x\in X$, there exists  an integer $\tilde{n}_0(x) \geq 2k\delta^{-1}\log C$ such that
\begin{align*}
\log f_n(T^m x) &\leq n\delta + (n-k+m)(\beta+\delta) - m(\beta-\delta)+\tilde{n}_0(x)\delta \\
 &\leq n\beta+2(n+m)\delta+\tilde{n}_0(x)\delta\leq n\alpha+4(n+m)\delta\\
 &\leq n \alpha +(n+m)\epsilon
\end{align*}
for all $n\geq \tilde{n}_0(x)$ and $m\in \N$. This together with (\ref{e-t}) yields (\ref{e-2.2}).
\end{proof}

As a direct corollary of Proposition \ref{lem-2.1}, we have the following.

\begin{cor}
\label{cor-2.1} Under the assumptions of Proposition \ref{lem-2.1}, for any
$\epsilon, N>0$ and for $\mu$-almost every $x\in X$, there is
$C(x)>0$ such that
$$|f_n(T^m x)|\leq C(x) \exp(n \max\{\alpha, -N\})\exp ((n+m)\epsilon)$$
for all $n,m\in \N$.
\end{cor}

\medskip
\begin{proof}[Proof of Theorem \ref{thm-1.2}] We only need to prove part (i), since parts (ii) and (iii)
follow immediately from (i), the variational principle (\ref{e-1.3}),
and Proposition \ref{pro-1.1}.

Fix an ergodic measure $\mu\in \M_\sigma(\Sigma)$. The direction
$M_*(\mu)\geq \max\{A_{*}^{(j)}(\mu):\; j\in \Lambda\}$ follows from
the fact that
$$
\|A_{i_1}^{(j,j)}\cdots A_{i_n}^{(j,j)}\|\leq \|T^{-1}M_{i_1}\cdots
M_{i_n}T\|\leq \|T^{-1}\|\|T\|\|M_{i_1}\cdots M_{i_n}\|
$$
for any $j\in \Lambda$ and $i_1,\ldots,i_n\in \{1,\ldots,\ell\}$. We
only need to prove the other direction.
% \begin{equation}
% \label{e-2.m}
% M_*(\mu)=\max\{ A^{(j)}_*(\mu):\; j=1,\ldots,s\},
% \end{equation}

By Furstenberg-Kesten's theorem \cite{FuKe60} on random matrices, or  Kingman's sub-additive ergodic theorem (see e.g. \cite{Wal-book}), we have for $\mu$-almost every $x=(x_i)_{i=1}^\infty\in \Sigma$,
 \begin{equation}
 \label{e-2.n}
 \lim_{n\to \infty}\frac{1}{n}\log \|M_{x_1\cdots x_n}\|=M_{*}(\mu).
 \end{equation}
For any $i \in \{ 1,\ldots,t \}$, define a sequence
$\{f_n^{(j)}\}_{n=1}^\infty$ of non-negative functions on $\Sigma$
by setting
$$
f_n^{(j)}(x)=\|A_{x_1}^{(j,j)}\cdots A_{x_n}^{(j,j)}\|
$$
for all $x=(x_i)_{i=1}^\infty\in \Sigma$.
Let $\epsilon,N>0$. Apply Corollary \ref{cor-2.1} for $\{f_n^{(j)}\}_{n=1}^\infty$ to obtain that,  for
$\mu$-almost every $x=(x_i)_{i=1}^\infty\in \Sigma$,  there exists
$C(x)\geq 1$ such that
 \begin{equation}
 \label{e-2.s}
 \begin{split}
 \|A^{(j,j)}_{x_{m+1}x_{m+2}\cdots x_{m+n}}\|&\leq C(x)\exp(n\max\{A^{(j)}_{*}(\mu),-N\})\exp((n+m)\epsilon)\\
 &\leq C(x) \exp(n\max\{W,-N\})\exp((n+m)\epsilon),
 \end{split}
  \end{equation}
 for all $j \in \{ 1,\ldots,t \}$ and $n,m\in \N$, where
 $$W=\max\{A^{(j)}_{*}(\mu):\; j \in \Lambda\}.$$
For the rest of the proof, we take a point $x=(x_i)_{i=1}^\infty\in \Sigma$
such that both (\ref{e-2.n}) and (\ref{e-2.s}) hold for $x$.

Fix $n \in \N$. According to \eqref{e-2.7_1}, $T^{-1}M_{x_1\cdots x_n}T$  is a partitioned matrix of the form $(B^{(j,k)})_{1\leq j,k\leq t}$, where each $B^{(j,k)}$ is a $d_j\times d_k$ matrix given by
 \begin{equation}
 \label{e-2.7}
 B^{(j,k)}=\sum_{j\leq y_1\leq y_2\leq \cdots\leq y_{n-1}\leq k} A_{x_1}^{(j,y_1)}A_{x_2}^{(y_1,y_2)}\cdots
 A_{x_n}^{(y_{n-1},k)}.
 \end{equation}
 It is easy to check that the number of words $y_1y_2\cdots y_{n-1}\in \{1,\ldots,t\}^{n-1}$, satisfying the restriction
 $j\leq y_1\leq y_2\leq \cdots\leq y_{n-1}\leq k$, is bounded above by
 $h(n)=(2n)^t$. Furthermore, each such a word $jy_1y_2\cdots y_{n-1}k$
 can be written as $a_1^{n_1}a_2^{n_2} \cdots a_s^{n_s}$,
 where  $s \in \{ 1,\ldots,t \}$, $j = a_1<\cdots<a_s = k$,
 and $n_1,\ldots,n_s\in \N$ with
 $n_1+\cdots+n_s=n+1$. Hence
 \begin{equation}
 \label{e-2.8'}
 A_{x_1}^{(j,y_1)}A_{x_2}^{(y_1,y_2)}\cdots
 A_{x_n}^{(y_{n-1},k)}=W_1 A_{x_{n_1}}^{(a_1,a_2)} W_2 A_{x_{n_1+n_2}}^{(a_2,a_3)}\cdots
 W_{s-1} A_{x_{n_1+n_2+\cdots+n_{s-1}}}^{(a_{s-1},a_s)}W_s,
 \end{equation}
where
\begin{align*}
  W_i = \begin{cases}
          {\bf I}_{d_{a_i} \times d_{a_i}} &\text{if } n_i=1, \\
          A^{(a_i,a_i)}_{x_{n_0+\cdots+n_{i-1}+1}} \cdots A^{(a_i,a_i)}_{x_{n_0+\cdots+n_i-1}} &\text{if } n_i>1
        \end{cases}
\end{align*}
for all $i \in \{ 1,\ldots,s \}$. Here ${\bf I}_{d \times d}$ is the $d \times d$ identity matrix and $n_0=0$.
Observe that (\ref{e-2.s}) gives
$$\|W_i\|\leq
C(x)\exp\left((n_i-1)\max\{W,-N\}\right)\exp\left((n_1+\cdots+n_i-1)\epsilon\right)
$$
for all $i \in \{ 1,\ldots,s \}$.
Hence, by (\ref{e-2.8'}), we have
 \begin{equation}
 \label{e-2.8}
 \begin{split}
\| A_{x_1}^{(j,y_1)} A_{x_2}^{(y_1,y_2)}\cdots
 &A_{x_n}^{(y_{n-1},k)}\| \leq  L^{s-1} \prod_{i=1}^s\|W_i\|\\
 &\leq L^{s-1} C(x)^s  \exp((n+1-s)\max\{W,-N\})\exp(ns\epsilon)\\
 &\leq D L^{t} C(x)^t  \exp(n\max\{W,-N\})\exp(nt\epsilon),
 \end{split}
 \end{equation}
 where
\begin{align*}
  L&=1+\max\{\|A^{(j_1,j_2)}_i\|:\; j_1,j_2 \in \{ 1,\ldots,t \} \text{ and } i \in \{ 1,\ldots,\ell\} \}, \\
  D&=\max\{1,\exp((t+1)\max\{W,-N\})\}.
\end{align*}
Therefore, by (\ref{e-2.7})--(\ref{e-2.8}),  we have the estimate
 \begin{align*}
 \|T^{-1}M_{x_1\cdots x_n}T\|&\leq
t^2\max\{ \|B^{(j,k)}\| : j,k \in \{ 1,\ldots,t \} \} \\
&\leq
 t^2 h(n) D L^{t} C(x)^t  \exp(n(\max\{W,-N\}))\exp(nt\epsilon)
 \end{align*}
for all $n\in \N$. Combining this estimate and (\ref{e-2.n}) yields
 $$
 M_*(\mu)=\lim_{n\to \infty}\frac{1}{n}\log \|T^{-1}M_{x_1\cdots x_n}T\|\leq \max\{W,-N\}+t\epsilon.
 $$
 Letting $N\to \infty$ and $\epsilon\to 0$, we get $$M_*(\mu)\leq W=\max\{A^{(j)}_{*}(\mu):\; j\in\Lambda\},$$
 which finishes the proof of part (i) of Theorem \ref{thm-1.2}.
 \end{proof}

\section{Extensions and remarks}
For an  invertible matrix $M\in \R^{d\times d}$, following \cite{Fal88}, we define the singular value
function of $M$ as
$$\phi^q(M) = \alpha_1(M)\cdots \alpha_k(M)\alpha_{k+1}(M)^{q-k},$$
where $ 0\leq q < d$, $k$ is the integral part of $q$, and $\alpha_i(M)$ is the $i$-th largest singular value of $M$. For $q>d$, we put $\phi^q(M) =
|\det(M)|^{q/d}$.
It is known (see \cite[Lemma 2.1]{Fal88}) that $\phi^q$ is sub-multiplicative in the sense that
$$\phi^q(M_1M_2)\leq \phi^q(M_1)\phi^q(M_2)$$ for any two invertible matrices $M_1,M_2 \in \R^{d\times d}$.   For a given family of
invertible matrices $\{M_i\}_{i=1}^\ell\subset \R^{d\times d}$, similar to (\ref{e-1.1}), we define
\begin{equation} \label{e-affine_pressure}
P^\phi(q)=\lim_{n\to \infty} \frac{1}{n}\log \sum_{J\in \Sigma_n}\phi^q(M_J).
\end{equation}
For $\mu\in \M_\sigma(\Sigma)$,
we define
\begin{equation} \label{e-energy}
\phi^q_*(\mu)=\lim_{n\to \infty}\frac{1}{n}\sum_{J\in \Sigma_n}\mu(J) \log \phi^q(M_J).
\end{equation}
Then by \cite[Theorem 2.6]{Kae04}, or more generally by \cite[Theorem 1.1]{CFH08},  we have the following variational principle
$$
P^\phi(q)=\max\{\phi^q_*(\mu)+h(\mu):\; \mu\in \M_\sigma(\Sigma)\}.
$$
Similarly we can study the structure of the equilibrium states of $P^\phi(q)$. It is easy to see that Theorem \ref{thm-1.2} remains true for $P^\phi(q)$ when $0\leq q\leq 1$ or $q\geq d-1$. Observe also that it is true when $q$ is an integer: if $M^{\wedge q}$ is the $q$-th exterior product of $M \in \R^{d \times d}$ (i.e.\ the $\binom{d}{q} \times \binom{d}{q}$ matrix whose entries are the $q \times q$ minors of $M$), then $$\alpha_1(M^{\wedge q})=\alpha_1(M)\cdots \alpha_{q}(M)=\phi^q(M).$$ This gives a partial answer to \cite[Question 6.3]{KV09}.

\begin{que}
  When using \eqref{e-affine_pressure} and \eqref{e-energy} instead of \eqref{e-1.1} and \eqref{e-1.2}, does something like Theorem \ref{thm-1.2} hold for $q \in [1,d-1] \setminus \N$?
\end{que}

We remark that some assumption was given in \cite{FaSl09} so that an analogue of \eqref{e-2.1} (where $\|\cdot\|$ is replaced by $\phi^q(\cdot)$) holds; and for such case, an analogue of Proposition \ref{pro-1.1} holds for $P^\phi$ (cf.\ \cite[Theorem 5.5]{Fen09b}).

%\section{Appendix}
%In this section, we cite some multifractal results about the
%topological Lyapunov exponents of matrices.

%For a given family of $d\times d$ matrices $\{M_i\}_{i=1}^\ell$, let
%\begin{equation}
%E(\alpha):=\left\{x=(x_i)_{i=1}^\infty\in \Sigma:\; \lim_{n\to \infty}
%\frac{1}{n}\log \|M_{x_1\ldots x_n}\|=\alpha \right\}, \quad
%\alpha\in \R.
%\end{equation}

%\begin{thm}
%\label{thm-4.1}
%\begin{itemize}
%\item[(i)] Assume that $\{M_i\}_{i=1}^\ell$ is irreducible on
%$\C^d$ and assume that $P(q)\neq \infty$ for $q<0$. Then
%$E(\alpha)\neq \emptyset$ if and only if $\lim_{q\to
%-\infty}P(q)/q\leq \alpha\leq \lim_{q\to \infty}P(q)/q$.
%Furthermore, $h_{\rm top}(E(\alpha))=\inf\{P(q)-\alpha q:\; q\in
%\R\}$.
%\item[(ii)] In the general case, for any $q>0$, if $\alpha=P'(q+)$
%or $\alpha=P'(q-)$, then $E(\alpha)\neq \emptyset$ and
%$$
%h_{\rm top}(E(\alpha))=\inf\{P(t)-\alpha t:\; t\in \R\}=P(q)-\alpha
%q.$$
%\end{itemize}
%\end{thm}
% Part (i) of the theorem   was proved in \cite{Fen09} and part
%(ii) was proved in \cite[Theorem 1.3]{FeHu09} in a more general
%setting. We remark that there exists counter-examples such that
%$h_{\rm top}(E(\alpha))<\inf\{P(t)-\alpha t:\; q\in \R\}$ for
%$\alpha\in (P'(q-), P'(q+))$, see \cite[Example 1.1]{Fen09}.

\bigskip

 \noindent {\bf Acknowledgements.} Feng was partially supported by  the RGC grant in the Hong Kong Special Administrative
Region, China.
K\"{a}enm\"{a}ki acknowledges the support of the Academy of
Finland (project \#114821). He also thanks the CUHK, where this research was started, for warm hospitality. The authors are grateful to Guivarc'h and Le Page for pointing out the reference \cite{GuLe04}.

\end{document}